\documentclass[sigconf]{acmart}

\usepackage{hyperref}
\usepackage{macros}
\usepackage{amscd}
\usepackage{xcolor}
\usepackage{microtype}
\usepackage{balance}
\usepackage{tabularx}

\definecolor{green}{rgb}{.08,.69,.10}
\definecolor{blue}{rgb}{.01,.26,.87}
\definecolor{red}{rgb}{.89,0,0}
\hypersetup{colorlinks,linkcolor={red!70!black},
  citecolor={green!70!black},urlcolor={blue!70!black}}

\newtheorem{theorem}{Theorem}
\newtheorem{proposition}[theorem]{Proposition}
\newtheorem{lemma}[theorem]{Lemma}
\newtheorem{corollary}[theorem]{Corollary}
\newtheorem{definition}[theorem]{Definition}

\def\BibTeX{{\rm B\kern-.05em{\sc i\kern-.025em b}\kern-.08emT\kern-.1667em\lower.7ex\hbox{E}\kern-.125emX}}

\copyrightyear{2019}
\acmYear{2019}
\setcopyright{acmlicensed}
\acmConference[ISSAC '19]{International Symposium on Symbolic and Algebraic Computation}{July 15--18, 2019}{Beijing, China}
\acmBooktitle{International Symposium on Symbolic and Algebraic Computation (ISSAC '19), July 15--18, 2019, Beijing, China}
\acmPrice{15.00}
\acmDOI{10.1145/3326229.3326251}
\acmISBN{978-1-4503-6084-5/19/07}

\begin{document}

\title{Standard Lattices of Compatibly Embedded Finite Fields}

\author{Luca De Feo}
\email{luca.de-feo@uvsq.fr}
\orcid{0000-0002-9321-0773}
\affiliation{%
  \institution{Universit\'e Paris Saclay -- UVSQ, LMV}
}
\author{Hugues Randriam}
\email{randriam@enst.fr}
\affiliation{%
  \institution{LTCI, T\'el\'ecom ParisTech}
}
\author{\'Edouard Rousseau}
\email{erousseau@enst.fr}
\affiliation{
  \institution{LTCI, T\'el\'ecom ParisTech}
  \institution{Universit\'e Paris Saclay -- UVSQ, LMV}
}

\begin{abstract}
  Lattices of compatibly embedded finite fields are useful in computer
  algebra systems for managing many extensions of a finite field
  $\F_p$ at once. %
  They can also be used to represent the algebraic closure $\bar\F_p$,
  and to represent all finite fields in a standard manner.

  The most well known constructions are Conway polynomials, and the
  Bosma--Cannon--Steel framework used in Magma. %
  In this work, leveraging the theory of the Lenstra-Allom\-bert
  isomorphism algorithm, we generalize both at the same time. %
  
  Compared to Conway polynomials, our construction defines a much
  larger set of field extensions from a small pre-computed table;
  however it is provably as inefficient as Conway polynomials if one
  wants to represent \emph{all} field extensions, and thus yields no
  asymptotic improvement for representing $\bar\F_p$.

  Compared to Bosma--Cannon--Steel lattices, it is considerably more
  efficient both in computation time and storage: all algorithms have
  at worst quadratic complexity, and storage is linear in the number
  of represented field extensions and their degrees.

  Our implementation written in C/Flint/Julia/Nemo shows that our
  construction in indeed practical.
\end{abstract}




\maketitle

\section{Introduction}
\label{sec:introduction}

Computer algebra systems (CAS) are often faced with the problem of
constructing several extensions of a finite field $\F_p$ in a
\emph{compatible} way, i.e., such that the (subfield) inclusion
lattice of the given extensions can be computed and evaluated
efficiently.

Concretely, what is sought is a data structure $\Lambda$ to represent
arbitrary collections of extensions of $\F_{p}$, in such a way that
elements of $\F_{p^m}$ are represented in optimal space (i.e., $O(m)$
coefficients), and that arithmetic operations are performed
efficiently (i.e., $O\left(\lcm(\ell,m)^d\right)$ arithmetic
operations to combine an element of $\F_{p^\ell}$ and an element of
$\F_{p^m}$, where $d\le 3$ and, possibly, $d=1+\varepsilon$). %
To this end, it is useful to set several sub-goals:

\begin{description}
\item[\emph{Effective embeddings:}] For any pair of extensions
  $k\subset K$ in $\Lambda$, there exists an efficiently computable
  embedding $\phi:k\to K$, and algorithms to evaluate $\phi$ on $k$,
  and the section $\phi^{-1}$ on $K$.
\item[\emph{Compatibility:}] The embeddings are \emph{compatible},
  i.e., for any triple $k\subset K\subset L$ in $\Lambda$, and
  embeddings $\phi:k\to K$, $\psi:K\to L$, $\chi:k\to L$, one has
  $\chi=\psi\circ\phi$.
\item[\emph{Incrementality:}] The data associated with an extension
  (e.g., its irreducible polynomial, change-of-basis matrices, \dots)
  must be computable efficiently and \emph{incrementally}, i.e.,
  adding a new field extension to $\Lambda$ does not require
  recomputing data for all extensions already in $\Lambda$. %
\item[\emph{Uniqueness:}] Any extension of $\F_{p}$ is determined by an
  irreducible polynomial whose definition only depends on the
  characteristic $p$ and the degree of the extension. %
\item[\emph{Generality:}] Extensions of $\F_{p}$ can be represented by
  arbitrary irreducible polynomials.
\end{description}

Some goals, such as incrementality, uniqueness and generality are
optional, and it is obvious that uniqueness and generality are even in
conflict with each other. %
An incremental data structure can be used to effectively represent an
algebraic closure $\bar\F_{p}$, with new finite extensions built on the
fly as they are needed. %
Uniqueness is useful for defining field elements in a standard way,
portable between different CAS, while generality is useful in a
context where the user is left with the freedom of choosing the
defining polynomials. %
Note that any solution can be made unique by replacing all random
choices with pseudo-random ones, however one is usually interested in
unique solutions that have a simple mathematical description. %
Also, any solution can be made general by means of an isomorphism
algorithm~\cite{LenstraJr91,Allombert02,brieulle2018computing,narayanan2016fast}. %
Other optional goals, such as computing normal bases or evaluating
Frobenius morphisms, may be added to the list, however they are out of
the scope of this work.

\paragraph{Previous work}
The first and most well known solution is the family of Conway
polynomials~\cite{Nickel1988,heath+loehr99}, first adopted in
GAP~\cite{GAP4}, and then also by Magma~\cite{MAGMA} and
Sage~\cite{Sage}. %
Conway polynomials yield uniqueness, however computing them requires
exponential time using the best known algorithm, thus incrementality is only available at a
prohibitive cost; for this reason, they are usually pre-computed and
tabulated up to some bound.

Lenstra~\cite{LenstraJr91} was the first to show the existence of a (incremental, general) data
structure computable in deterministic polynomial time. He proved that, besides the problem of
finding irreducible polynomials, any other question is amenable to
linear algebra. %
Subsequent work of Lenstra and de
Smit~\cite{lenstra+desmit08-stdmodels} tackled the uniqueness problem,
albeit only from a theoretical point of view. %

In practice, randomized algorithms are good enough for a CAS, then
polynomial factorization and basic linear algebra provide an easy
(incremental, general) solution, that was first analyzed by Bosma,
Cannon and Steel~\cite{bosma+cannon+steel97}, and is currently used in
Magma. %

All solutions presented so far have superquadratic complexity, i.e.,
$d>2$. %
Recent work on embedding algorithms~\cite{DoSc12,DeDoSc13,DeDoSc2014}
yields subquadratic (more precisely, $d\le 1.5$) solutions for
specially constructed (non-unique, non-general) families of
irreducible polynomials, and even quasi-optimal ones (i.e.,
$d=1+\varepsilon$) if a quasi-linear modular composition algorithm is
available. %
However these constructions involve counting points of random elliptic
curves over finite fields, and have thus a rather high polynomial
dependency in $\log p$; for this reason, they are usually considered
practical only for relatively small characteristic.

\paragraph{Our contribution}
In this work we present an incremental, general and/or unique solution
for lattices of compatibly embedded finite fields, where all
embeddings can be computed and evaluated in quasi-quadratic time. %
Our starting point is Allombert's~\cite{Allombert02} and
subsequent~\cite{brieulle2018computing} improvements to Lenstra's
isomorphism algorithm~\cite{LenstraJr91}. %
Plugging them in the Bosma--Cannon--Steel framework immediately produces
an incremental general solution with quasi-quadratic complexity;
however we go much further. %
Indeed, we show that the compatibility requirement can be taken a step
further by constructing a lattice of $\F_{p}$-algebras with a
distinguished element, which is a byproduct of the Lenstra-Allombert
algorithm.

The advantages of our construction over a naive combination of the
Lenstra-Allombert algorithm and the Bosma--Cannon--Steel framework are
multiple. %
Storage drops from quadratic to linear in the number of extensions
stored in $\Lambda$ and in their degrees, and the cost of adding a new
extension to $\Lambda$ drops similarly. %

Our $\F_{p}$-algebras are constructed by tensoring an arbitrary
lattice of extensions of $\F_{p}$ with what we call a \emph{cyclotomic
  lattice} (see next section). %
In this work we mostly abstract away from the concrete instantiation
of the cyclotomic lattice, only fixing a choice in
Section~\ref{sec:implementation}, where we use Conway polynomials to
analyze the complexity and implement our algorithms. %
This choice allows us to uniquely represent finite fields of degree
exponentially larger than with Conway polynomials alone; however it
also has the serious drawback of being generically as hard to compute
as Conway polynomials, and thus relatively unpractical. %
We leave the exploration of other, more practical, instantiations of
cyclotomic lattices for future work.

\paragraph{Organization}
The presentation is structured as follows. %
In Section~\ref{sec:conway} we review some basic algorithms and facts
on roots of unity and Conway polynomials. %
In Section~\ref{sec:lenstra} we review the Lenstra-Allombert algorithm
and we define and study \emph{Kummer algebras}, the main ingredient to
our construction. %
In Section~\ref{sec:compatibleH90} we introduce a notion of
compatibility for solutions of Hilbert~90 in Kummer algebras, that
provides standard defining polynomials for finite fields.  Then in
Section~\ref{sec:construction} we again use these compatible solutions
to construct standard compatible embeddings between finite fields,
from which a lattice can be incrementally constructed.  Finally, in
Section~\ref{sec:implementation} we give the complexities of our
algorithms, and present our implementation.

\section{Preliminaries}
\label{sec:conway}

\paragraph{Fundamental algorithms and complexities}
Throughout this paper we let $\F_p$ be a finite field. For simplicity,
we shall assume that $p$ is prime, which is arguably the most useful
case, however our results could easily be extend to non-prime
fields. We measure time complexities as a number of arithmetic
operations $(+,\times,/)$ over $\F_p$, and storage as a number of
elements of $\F_p$. %
We let $\M(m)$ denote the number of operations required to multiply
two polynomials with coefficients in $\F_p$ of degree at most $m$, and
adopt the usual super-linearity assumptions on the function $\M$
(see~\cite[Ch.~8.3]{vzGG}).

Any finite extension $\F_{p^m}$ can be represented as the quotient of
$\F_p[X]$ by an irreducible polynomial of degree $m$. %
The algorithms we present in the next sections need not assume any
particular representation for finite fields, however when analyzing
their complexities we will assume this representation. %
Then, multiplications in $\F_{p^m}$ can be carried out using
$O(\M(m))$ operations, and inversions using $O(\M(m)\log(m))$. %
In this work we will also need to perform computations in algebras
$\F_{p^m}\otimes\F_{p^n}$: representing them as quotients of a
bivariate polynomial ring, we can multiply elements using $O(\M(mn))$
operations.

We will extensively use a few standard routines, that we recall
briefly. %
Brent and Kung's algorithm~\cite{brent+kung} computes the
\textbf{modular composition} $f(g)\bmod h$ of three polynomials
$f,g,h\in\F_p[X]$ of degree at most $m$ using
$O\bigl(m^{(\omega+1)/2)}\bigr)$ operations, where $\omega$ is the
\emph{exponent of linear algebra} over $\F_p$. %
The Kedlaya--Umans algorithm~\cite{KeUm11} solves the same problem,
and has better complexity in the binary RAM model, however it is
widely considered impractical, we shall thus not consider it in our
complexity estimates.

By applying \emph{transposition
  techniques}~\cite{burgisser+clausen-shokrollahi,bostan+lecerf+schost:tellegen}
to Brent and Kung's algorithm, Shoup~\cite{shoup94,shoup99} derived an
algorithm to compute \textbf{minimal polynomials} of arbitrary
elements of $\F_{p^m}$, having the same complexity
$O\bigl(m^{(\omega+1)/2)}\bigr)$. %
Kedlaya and Umans' improvements also apply to Shoup's minimal
polynomial algorithm, with the same practical limitations.

Shoup's techniques can also be applied to \textbf{evaluate embeddings
  of finite fields}. %
Let $\F_{p^\ell},\F_{p^m}$ be a pair of finite fields related by an
embedding $\phi:\F_{p^\ell}\hookrightarrow\F_{p^m}$, and let a
generator $\alpha_\ell$ of $\F_{p^\ell}$ and its image
$\phi(\alpha_\ell)$ in $\F_{p^m}$ be given. %
Given these data, for any element $x\in\F_{p^\ell}$ it is possible to
compute its image $\phi(x)$ using $O\bigl(m^{(\omega+1)/2)}\bigr)$
operations; similarly, given an element $y=\phi(x)$ in $\F_{p^m}$, it
is possible to recover $x$ in the same asymptotic number of
operations. %
The relevant algorithms are summarized in~\cite[Sec.~6]{ffisom-long};
note that, for specially constructed generators $\alpha_\ell$, more
efficient algorithms may
exist~\cite{df+schost12,DoSc12,DeDoSc13,DeDoSc2014}.

The present work focuses on algorithms to \textbf{compute embeddings
  of finite fields}, i.e., algorithms that, given finite fields
$\F_{p^\ell}$ and $\F_{p^m}$ with $\ell\,|\,m$, find an embedding
$\phi:\F_{p^\ell}\hookrightarrow\F_{p^m}$, a generator $\alpha_\ell$
of $\F_{p^\ell}$, and its image $\phi(\alpha_\ell)$. %
An extensive review of known algorithms is given
in~\cite{brieulle2018computing}; here we shall only be interested in
the Lenstra--Allombert isomorphism
algorithm~\cite{LenstraJr91,Allombert02}, and its adaptation to
compatible lattices of finite fields.

\paragraph{Conway polynomials and cyclotomic lattices}
The algorithms of the next sections will be dependent on the
availability of a \emph{cyclotomic lattice}. %
By this we mean, formally, a collection
\[ \system=\{(K_\ell,\zeta_\ell)\}_{\ell\in I} \]
over some support set $I\subset\N\setminus p\N$.
Where $K_\ell$ is an explicitly represented finite extension of $\F_p$,
and $\zeta_\ell\in K_\ell$ a generating element that is also a primitive $\ell$-th root
of unity,
so \[ K_\ell=\F_p(\zeta_\ell),\quad(\zeta_\ell)^\ell=1, \]
together with
explicit embeddings
\[
\begin{array}{cccc}
  \embedcyc{\ell}{m}: & K_\ell & \hookrightarrow & K_m \\
  & \zeta_\ell & \mapsto & (\zeta_{m})^{\frac{m}{\ell}}
\end{array}
\]
whenever $l\,|\,m$.

If $I$ is a finite set of indices, there is an easy randomized
algorithm to construct a cyclotomic lattice: compute
$n=\lcm_{\ell\in I}(\ell)$, construct the smallest field $\F_{p^a}$ such
that $n$ divides $p^a-1$, take a random $x^{(p^a-1)/n}\in\F_{p^a}$ and
test that it has multiplicative order $n$;
then all roots $\zeta_\ell$ in the
lattice are constructed as powers of this element,
and we can set $K_\ell=\F_p(\zeta_\ell)\subset\F_{p^a}$
and let $\embedcyc{\ell}{m}$ be natural inclusion.

Nevertheless, the most useful cyclotomic lattices are those where $I$
is the whole set $\N\setminus p\N$.
It may seem odd to ask for such data, which in the end provides a representation of $\bar\F_p$,
as a prerequisite for a construction whose goal is precisely to represent $\bar\F_p$.
However we shall see that a relatively small cyclotomic sub-lattice is
enough to construct a much larger lattice of compatibly embedded
finite fields; it thus makes sense to assume that a cyclotomic lattice
is available, if it can be computed \emph{incrementally}.

\emph{Conway polynomials}~\cite{Nickel1988} offer a classic example of
cyclotomic lattice. %
The $a$-th Conway polynomial $C_a\in\F_p[X]$ is defined as the
\emph{lexicographically smallest} monic irreducible polynomial of
degree $a$ that is also \emph{primitive} (i.e., its roots generate
$\F_{p^a}^\times$) and \emph{norm compatible} (i.e,
$$C_a\Bigl(X^{\frac{p^b-1}{p^a-1}}\Bigr) = 0 \mod C_b$$
whenever $a$ divides $b$). %
The cyclotomic lattice is defined first by letting $\zeta_{p^a-1}$ be the
image of $X$ in $K_{p^a-1}=\F_{p^a}=\F_p[X]/C_a$ for any $a$;
this is then extended to all $\ell\in I$ by setting $K_\ell=K_{p^a-1}$
and $\zeta_\ell=(\zeta_{p^a-1})^{\frac{p^a-1}{\ell}}$
where $\F_{p^a}$ is the smallest extension $\F_p$ containing $\ell$-th roots of unity.

The best known algorithm to compute Conway polynomials has exponential
complexity~\cite{heath+loehr99}, hence they are usually precomputed
and tabulated up to a certain bound. %
Most computer algebra systems switch to other ways of representing
finite fields when the tables of Conway polynomials are not enough. %
A notable exception is SageMath~\cite{Sage} (since version
5.13~\cite{Roe2013}), that defines \emph{pseudo-Conway polynomials} by
relaxing the ``lexicographically first'' requirement; although
easier to compute in practice, their computation still requires an
exponential amount of work.

Other ways to construct cyclotomic lattices are possible. %
One may, for example, factor cyclotomic polynomials over $\F_p$, being
careful to maintain compatibility. %
In the next sections we shall not suppose any cyclotomic lattice
construction in particular, and simply assume that we are given a
collection $\system$ that satisfies the properties given at the
beginning of this paragraph.

\section{The Lenstra-Allombert algorithm}
\label{sec:lenstra}

We now review the theory behind Allombert's
adaptation~\cite{Allombert02} of Lenstra's isomorphism
algorithm~\cite{LenstraJr91}. %
This will be our stepping stone towards the definition of some
\emph{standard} elements in field extensions of $\F_p$ with an
effective compatibility condition.

The main ingredient of the algorithm is an extension of Kummer
theory. %
Because of this, the algorithm is limited to field extensions
$\F_{p^\ell}$ of degree $\ell$ prime to $p$. %
The easier case of extensions of degree $p^e$ is covered in a similar
way using Artin-Schreier theory, and the generic case is solved by
separately computing isomorphisms for the power-of-$p$ and the
prime-to-$p$ parts, and then tensoring the results together. %
Due to space constraints, we will not give details for the general
case here; see~\cite{LenstraJr91,Allombert02,brieulle2018computing}.

For any finite extension of $\F_{p}$, we denote by $ \sigma:x\mapsto x^p $
the Frobenius automorphism.
Let $\ell$ be an integer not divisible by $p$.
Then $\sigma$ is an $\F_{p}$-linear endomorphism of $\F_{p^\ell}$
with minimal polynomial $T^\ell-1$, separable but not necessarily split,
i.e., $\F_{p^\ell}$ is not necessarily a Kummer extension of $\F_{p}$.
We extend scalars and work in the
\emph{Kummer algebra of degree~$\ell$}:
\[
  A_\ell = \F_{p^\ell}\otimes\F_{p}(\zeta_\ell),
\]
where $\otimes$ is the tensor product over $\F_{p}$, and $\zeta_\ell$ is a primitive
$\ell$-th root of unity, taken from the given cyclotomic lattice. We
call $\F_{p}(\zeta_\ell)$ the \emph{field of scalars} of $A_\ell$, and
we define the \emph{level} of $A_\ell$ as
\[
  \nu(\ell) = \mathrm{ord}_{(\mathbb{Z}/\ell\mathbb{Z})^\times}(p) = [\F_{p}(\zeta_\ell):\F_{p}],
\]
that is, the degree of its field of scalars.

Now $\sigma\otimes1$ is a $1\otimes\F_{p}(\zeta_\ell)$-linear endomorphism of $A_\ell$
with $\ell$ distinct eigenvalues, namely the powers of $1\otimes\zeta_\ell$.
Thus, if $\eta=\zeta_\ell^i$ is any $\ell$-th root of unity in $\F_{p}(\zeta_\ell)$,
the corresponding eigenspace is defined by the \emph{Hilbert 90 equation for $\eta$}:
\begin{equation}
  \tag{H90}
 (\sigma\otimes1)(x) = (1\otimes\eta)x,
  \label{h90}
\end{equation}
which plays the role of $\sigma(x)=\eta x$ in classical Kummer theory.
The solutions of~\eqref{h90} in $A_\ell$ form a
$1\otimes\F_{p}(\zeta_{\ell})$-vector space of dimension $1$,
and if $x$ is such a solution for $\eta$, then $x^j$ is a solution for $\eta^j$.

In particular, let $\alpha_\ell$ be a nonzero solution of~\eqref{h90} for $\zeta_\ell$.
Then $1,\alpha_\ell,\cdots,(\alpha_\ell)^{\ell-1}$ are eigenvectors for distinct eigenvalues and thus
form a basis of $A_\ell$ over $1\otimes\F_{p}(\zeta_{\ell})$.
Likewise
\[ (\alpha_\ell)^\ell = 1\otimes\a_\ell \]
for some scalar $\a_\ell\in\F_{p}(\zeta_\ell)$,
that we shall call the \emph{Kummer constant} of $\alpha_\ell$.
This proves:
\begin{proposition}
\label{alphagen}
Any nonzero solution $\alpha_\ell$ of \eqref{h90} for $\zeta_\ell$ is a generating element for $A_\ell$ as an algebra over $1\otimes\F_{p}(\zeta_{\ell})$,
inducing an isomorphism
\[ A_\ell\simeq \F_{p}(\zeta_{\ell})[T]/(T^\ell-\a_\ell). \]
\end{proposition}

Since $A_\ell$ is known to be an \'etale algebra, 
$\alpha_\ell$ being nonzero implies that $\a_\ell$ is nonzero, which in turn implies
that $\alpha_\ell$ is \emph{invertible} in $A_\ell$, indeed $\;(\alpha_\ell)^{-1} = (1\otimes\a_\ell^{-1})(\alpha_\ell)^{\ell-1}$.


We will make frequent use of the following:
\begin{lemma}
\label{FrobFrob}
Let $K,L$ be two finite extensions of $\F_{p}$.
Then, for any $\beta\in K\otimes L$,
we have $(\sigma\otimes\sigma)(\beta)=\beta^p$.
\end{lemma}
\begin{proof}
If $\beta=u\otimes v$ is an elementary tensor we have $(\sigma\otimes\sigma)(\beta)=u^p\otimes v^p=\beta^p$.
This then extends by linearity since we're in characteristic $p$. 
\end{proof}
In this generality we also introduce the following notation: if $\eta\in L$ has degree $d$ over $\F_{p}$,
then any $\beta\in K\otimes\F_{p}(\eta)\subset K\otimes L$ decomposes uniquely as $\beta = \sum_{i =  0}^{d-1}y_i\otimes\eta^i$,
and we set  \[ \first{\beta}{\eta}=y_0. \]
In particular, coming back to $A_\ell$, if we write
\[ \alpha_\ell = \sum_{i=0}^{a-1}x_i\otimes\zeta_{\ell}^i \] where $a=\nu(\ell)$,
it is shown in~\cite{Allombert02} that $x_0=\first{\alpha_\ell}{\zeta_\ell}$ is a \emph{generating element} for the extension
$\F_{p^\ell}/\F_{p}$.
Moreover, \cite{Allombert02} (see also~\cite{brieulle2018computing}) provides the following equations
that allow, in the opposite direction, to recover $\alpha_\ell$ from $x_0$:
\begin{equation}
\label{recoveralpha}
\begin{array}{l}
x_{a-1}=\sigma(x_0)/b_0\\
x_{i}\quad\,=\sigma(x_{i+1})-b_{i+1}x_{a-1}\quad\text{for $i=a\!-\!2$ down to $1$}
\end{array}
\end{equation}
where $\zeta_\ell^a=\sum_{i=0}^{a-1}b_i\zeta_\ell^i$ is the minimal equation for $\zeta_\ell$.

\begin{proposition}
\label{depend}
With the notations above, there are precisely $\ell$ elements $x\in A_\ell$ that are solutions of~\eqref{h90} for $\zeta_\ell$
and satisfy $x^\ell=1\otimes\a_\ell$, namely, these are the $(1\otimes\zeta_\ell)^u\alpha_\ell=(\sigma^u\otimes 1)(\alpha_\ell)$
for $0\leq u<\ell$.
The corresponding generating elements for $\F_{p^\ell}/\F_{p}$ are the $\first{(\sigma^u\otimes 1)(\alpha_\ell)}{\zeta_\ell}=\sigma^u(x_0)$;
they all have the same minimal polynomial, which is a generating polynomial for $\F_{p^\ell}/\F_{p}$ depending only on $\a_\ell$.
\end{proposition}
\begin{proof}
The solutions of~\eqref{h90} for $\zeta_\ell$ form a $1\otimes\F_{p}(\zeta_{\ell})$-vector space of dimension $1$,
thus they all are of the form $x=(1\otimes\xi)\alpha_\ell$.
Adding the condition $x^\ell=1\otimes\a_\ell$ then forces $\xi^\ell=1$,
from which all assertions follow.
\end{proof}

Now we consider Kummer algebras of various degrees.
Since we assumed that the fields of scalars are defined from a cyclotomic lattice $\system$,
they are compatibly embedded:
for $\ell\,|\,m$ prime to $p$, we have the embedding
\[
\begin{array}{cccc}
  \embedcyc{\ell}{m}: & \F_{p}(\zeta_\ell) & \hookrightarrow & \F_{p}(\zeta_m) \\
  & \zeta_\ell & \mapsto & (\zeta_{m})^{\frac{m}{\ell}}.
\end{array}
\]

It is easily shown that, as an $\F_{p}$-algebra, $A_\ell$ is isomorphic to a product of copies of $\F_{p^\ell}(\zeta_\ell)$,
and $A_m$ to a product of copies of $\F_{p^m}(\zeta_m)$.
This allows us to describe all $\F_{p}$-algebra morphisms from $A_\ell$ to $A_m$. However here we will focus only on a certain
subclass of them:
\begin{definition}
\label{Kembedding}
A \emph{Kummer embedding} of $A_\ell$ into $A_m$ is an \emph{injective} $\F_{p}$-algebra morphism $\Phi:A_\ell\hookrightarrow A_m$
such that:
\begin{itemize}
\item $\Phi$ extends the scalar embedding $1\otimes\embedcyc{\ell}{m}$
\item $\Phi$ commutes with $\sigma\otimes1$.
\end{itemize}
\end{definition}

\begin{proposition}
\label{Phialpha}
Let $\alpha_\ell\in A_\ell$ be a nonzero solution of~\eqref{h90} for $\zeta_\ell$, with Kummer constant $\a_\ell$.
Then, there is a $1$-to-$1$ correspondence between Kummer embeddings $\Phi:A_\ell\hookrightarrow A_m$ and solutions $\hat{\alpha}\in A_m$
of~\eqref{h90} for $(\zeta_m)^{\frac{m}{\ell}}$ that satisfy $(\hat{\alpha})^\ell=1\otimes\embedcyc{\ell}{m}(\a_\ell)$,
given by \[ \Phi\quad\longleftrightarrow\quad\hat{\alpha}=\Phi(\alpha_\ell). \]
\end{proposition}
\begin{proof}
Direct consequence of Proposition~\ref{alphagen} and Definition~\ref{Kembedding}.
\end{proof}

Actually, Kummer embeddings are easily characterized:
\begin{proposition}
\label{Phiphi}
There is a natural $1$-to-$1$ correspondence between Kummer embeddings $\Phi:A_\ell\hookrightarrow A_m$
and embeddings of finite fields $\phi:\F_{p^\ell}\hookrightarrow\F_{p^m}$, given by:
\begin{itemize}
\item If $\Phi$ is a Kummer embedding, then $\Phi$ maps $\F_{p^\ell}\otimes1$ into $\F_{p^m}\otimes 1$.
Thus the restriction of $\Phi$ to $\F_{p^\ell}\otimes1$ is of the form $\phi\otimes1$ for some $\phi:\F_{p^\ell}\hookrightarrow\F_{p^m}$.
\item Conversely, if $\phi:\F_{p^\ell}\hookrightarrow\F_{p^m}$ is an embedding of finite fields, then $\Phi=\phi\otimes\embedcyc{\ell}{m}$
is a Kummer embedding.
\end{itemize}
Moreover, this correspondence commutes with composition of embeddings.
\end{proposition}
\begin{proof}
Let $\Phi$ be a Kummer embedding. Being a $\F_{p}$-algebra morphism, it satisfies $\Phi(\beta^p)=\Phi(\beta)^p$ for all $\beta\in A_\ell$.
By Lemma~\ref{FrobFrob}, this means that $\Phi$ commutes with $\sigma\otimes\sigma$, and thus,
also with $(\sigma\otimes 1)^{-1}\circ(\sigma\otimes\sigma)=1\otimes\sigma$.
This implies that $\Phi$ maps $\F_{p^\ell}\otimes1$ into $\F_{p^m}\otimes 1$.
The other assertions are clear.
\end{proof}

\begin{corollary}
\label{alphaphi}
Let $\alpha_\ell\in A_\ell$ be a nonzero solution of~\eqref{h90} for $\zeta_\ell$,
with Kummer constant $\a_\ell$, and
let $\hat{\alpha}\in A_m$ be a solution of~\eqref{h90} for $(\zeta_m)^{\frac{m}{\ell}}$
that satisfies $(\hat{\alpha})^\ell=1\otimes\embedcyc{\ell}{m}(\a_\ell)$.
Then:
\begin{itemize}
\item
  $\hat{\alpha}\in\F_{p^m}\otimes\F_{p}((\zeta_m)^{\frac{m}{\ell}})\;\subset\;A_m$;
\item the assignation $\first{\alpha_{\ell}}{\zeta_\ell}\mapsto\first{\hat{\alpha}}{(\zeta_m)^{\frac{m}{\ell}}}$
defines an embedding $\phi:\F_{p^\ell}\hookrightarrow\F_{p^m}$;
\item $\Phi=\phi\otimes\embedcyc{\ell}{m}$ is the unique Kummer embedding such that $\Phi(\alpha_\ell)=\hat{\alpha}$.
\end{itemize}
\end{corollary}
\begin{proof}
By Proposition~\ref{Phialpha} there is a unique Kummer embedding $\Phi$ such that $\Phi(\alpha_\ell)=\hat{\alpha}$.
By Proposition~\ref{Phiphi} we have that $\Phi=\phi\otimes\embedcyc{\ell}{m}$ for some $\phi:\F_{p^\ell}\hookrightarrow\F_{p^m}$.
Writing $\alpha_\ell=\sum_{i=0}^{a-1}x_i\otimes\zeta_{\ell}^i$,
it follows that \[ \hat{\alpha}=\Phi(\alpha_\ell)=\sum_{i=0}^{a-1}\phi(x_i)\otimes(\zeta_m)^{\frac{mi}{\ell}}. \]
Thus $\first{\hat{\alpha}}{(\zeta_m)^{\frac{m}{\ell}}}=\phi(x_0)=\phi(\first{\alpha_{\ell}}{\zeta_\ell})$,
and, since $\first{\alpha_{\ell}}{\zeta_\ell}$ generates $\F_{p^\ell}$, this uniquely characterizes $\phi$.
\end{proof}

We can now state Allombert's algorithm and prove its correctness; we
give below a minor variation on the original algorithm, better adapted
to our more general setting.
\begin{algorithm}
  \caption{(Allombert's algorithm)}
  \label{algo:allombert}
  \begin{algorithmic}[1]
    \REQUIRE {$\F_{p^\ell}, \F_{p^m}$, for $\ell\,|\,m$ integers prime to $p$,
      and a cyclotomic lattice $\system[\{l,m\}]$.}
    \ENSURE {$s\in\F_{p^\ell}, t\in\F_{p^m}$, such that the assignation $s\mapsto t$
      defines an embedding $\phi:\F_{p^\ell}\hookrightarrow\F_{p^m}$.}
  \STATE Prepare the Kummer algebras $A_\ell$ and $A_m$.
  \STATE Find $\alpha_\ell\in A_\ell$ and $\alpha_m\in A_m$, nonzero solutions of \eqref{h90} for $\zeta_\ell$
  and $\zeta_m$ respectively.
  \STATE Compute their Kummer constants: $(\alpha_\ell)^\ell=1\otimes\a_\ell$ and
  $(\alpha_m)^m=1\otimes\a_m$.
  \STATE Compute $\kappa$, a $\ell$-th root of $\embedcyc{\ell}{m}(\a_\ell)/\a_m$.
  \STATE Return $\first{\alpha_{\ell}}{\zeta_\ell}$ and $\first{(1\otimes\kappa)(\alpha_m)^{\frac{m}{\ell}}}{(\zeta_m)^{\frac{m}{\ell}}}$.
  \end{algorithmic}
\end{algorithm}
\begin{proposition}
  Algorithm~\ref{algo:allombert} is correct: it returns elements that define an
  embedding $\phi:\F_{p^\ell}\hookrightarrow\F_{p^m}$.
\end{proposition}
\begin{proof}
By Propositions~\ref{Phialpha} and~\ref{Phiphi}, there exists $\hat{\alpha}\in A_m$
solution of~\eqref{h90} for $(\zeta_m)^{\frac{m}{\ell}}$ that satisfies $(\hat{\alpha})^\ell=1\otimes\embedcyc{\ell}{m}(\a_\ell)$.
On the other hand, $(\alpha_m)^{\frac{m}{\ell}}\in A_m$ is also a solution of~\eqref{h90} for $(\zeta_m)^{\frac{m}{\ell}}$,
thus $\hat{\alpha}=(1\otimes\lambda)(\alpha_m)^{\frac{m}{\ell}}$ for some $\lambda\in\F_{p}(\zeta_m)$.
It follows that $\embedcyc{\ell}{m}(\a_\ell)/\a_m=\lambda^\ell$ is a $\ell$-th power,
and $\kappa = (\zeta_{m})^{\frac{um}{\ell}}\lambda$ for some integer $u$.
Now we can replace $\hat{\alpha}$
with $(1\otimes(\zeta_{m})^{\frac{um}{\ell}})(\hat{\alpha})=(1\otimes\kappa)(\alpha_m)^{\frac{m}{\ell}}$
and conclude with Corollary~\ref{alphaphi}.
\end{proof}
From this proof and Proposition~\ref{depend}, it follows that another choice of the $l$-th root $\kappa$
only changes $\phi$ by a power of $\sigma$.
    
\section{Standard solutions of (H90)}
\label{sec:compatibleH90}

Plugging Algorithm~\ref{algo:allombert} into the Bosma--Cannon--Steel
framework immediately gives a way to compatibly embed arbitrary finite
fields. %
However, there are two points in Allombert's algorithm on which we
would like to improve:
\begin{description}
\item[\emph{Uniqueness:}] As mentioned, the element $\first{\alpha_{\ell}}{\zeta_\ell}$ is a generating element for $\F_{p^{\ell}}$,
or equivalently, it provides a defining irreducible polynomial of degree $\ell$.
However this polynomial depends on the choice of $\alpha_{\ell}$
(even though only through~$\a_\ell$, cf. Proposition~\ref{depend}).
\item[\emph{Compatibility:}] The embedding $\phi$ depends on a constant $\kappa$,
which itself depends on the choice of $\alpha_\ell,\alpha_m$ (and of a $\ell$-th root extraction).
Thus, given a certain number of finite fields, in order to ensure compatibility of the various
embeddings between them, one has to keep track of these constants $\kappa$
for all pairs $(\ell,m)$,
which grow \emph{quadratically} with the number of fields.
\end{description}

It would be useful if one could force $\kappa=1$, that is, if $\alpha_\ell,\alpha_m$
and $\Phi:A_\ell\hookrightarrow A_m$
could be chosen so that $\Phi(\alpha_\ell)=(\alpha_m)^{\frac{m}{\ell}}$.
From the description of Algorithm~\ref{algo:allombert},
this requires $\a_m=\embedcyc{\ell}{m}(\a_\ell)$.
Thus,
necessarily $\a_m$ lies in the subfield $\F_{p}((\zeta_m)^{\frac{m}{\ell}})$ of $\F_{p}(\zeta_m)$.
Possibly this condition could fail if $A_\ell$ and $A_m$ do not have the same field of scalars.
This motivates:
\begin{definition}
\label{complete}
A Kummer algebra is \emph{complete} if it is of the largest degree for a given level.
\end{definition}
Thus, the complete Kummer algebra of level $a$ is
\[ A_{p^a-1}=\F_{p^{p^a-1}}\otimes\F_{p^a} \]
with field of scalars $\F_{p^a}=\F_{p}(\zeta_{p^a-1})$ given by the corresponding $\zeta_{p^a-1}$ in our cyclotomic lattice $\system$,
e.g., defined by a (pseudo)-Conway polynomial of degree $a$.
\begin{lemma}
\label{Kummer_bizarre}
All nonzero solutions $\alpha_{p^a-1}\in A_{p^a-1}$ of \eqref{h90} for $\zeta_{p^a-1}$
have the same Kummer constant $\a_{p^a-1}=(\zeta_{p^a-1})^a$.
\end{lemma}
\begin{proof}
  From Lemma~\ref{FrobFrob} and the fact that $\sigma^a$ is trivial on
  $\F_{p^a}\simeq\F_{p}(\zeta_{p^a-1})$ we get that
\begin{equation*}
\begin{split}
(\alpha_{p^a-1})^{p^a}=(\sigma^a\otimes\sigma^a)(\alpha_{p^a-1})&=(\sigma^a\otimes1)(\alpha_{p^a-1})\\
&=(1\otimes\zeta_{p^a-1})^a\alpha_{p^a-1}.
\end{split}
\end{equation*}
We conclude since $\alpha_{p^a-1}$ is invertible.
\end{proof}
\begin{definition}
\label{alphastandard}
Let $l$ be an integer prime to $p$.
We define the \emph{standard} Kummer constant of order $\ell$ as
\[ \stda_{\ell}=(\embedcyc{l}{p^a-1})^{-1}((\zeta_{p^a-1})^a)\in\F_p(\zeta_\ell) \]
where $a=\nu(\ell)$ is the level of $A_\ell$.

We say a solution $\alpha_\ell\in A_\ell$ of~\eqref{h90} for $\zeta_\ell$ is
\emph{standard} if its Kummer constant is standard:
\[ (\alpha_\ell)^\ell=1\otimes\stda_{\ell}. \]
Then, by a \emph{decorated} Kummer algebra we mean a pair
\[ (A_\ell,\alpha_{\ell}) \]
with such $\alpha_\ell$ standard.
\end{definition}
Observe that $\embedcyc{l}{p^a-1}$ is an isomorphism when $a=\nu(\ell)$,
so $\stda_\ell$ is well defined.

For complete algebras, Lemma~\ref{Kummer_bizarre} asserts that all nonzero $\alpha_{p^a-1}$ are standard.
\begin{proposition}
\label{standardexiste}
Let $l$ be an integer not divisible by $p$.
Then $A_\ell$ can be decorated, i.e., it admits a standard $\alpha_\ell$.
Moreover, this $\alpha_\ell$ is unique up to a $\ell$-th root of unity.
\end{proposition}
\begin{proof}
Let $\alpha'_\ell$ be any nonzero solution of~\eqref{h90} for $\zeta_\ell$.
Set $a=\nu(\ell)$, pick any $\alpha_{p^a-1}\in A_{p^a-1}$ standard (Lemma~\ref{Kummer_bizarre}),
and pick any Kummer embedding $\Phi:A_\ell\hookrightarrow A_{p^a-1}$ (Proposition~\ref{Phiphi}).
Then $\Phi(\alpha'_\ell)$ and $(\alpha_{p^a-1})^{\frac{p^a-1}{\ell}}$ are two nonzero solutions of~\eqref{h90}
for $\embedcyc{l}{p^a-1}(\zeta_\ell)=(\zeta_{p^a-1})^{\frac{p^a-1}{\ell}}$ in $A_{p^a-1}$, thus there is a scalar $\lambda\in\F_{p}(\zeta_{p^a-1})$
such that
\begin{equation*}
(\alpha_{p^a-1})^{\frac{p^a-1}{\ell}}=(1\otimes\lambda)\Phi(\alpha'_\ell)=\Phi((1\otimes\tilde\lambda)\alpha'_\ell),
\end{equation*}
where $\tilde\lambda=(\embedcyc{l}{p^a-1})^{-1}(\lambda)\in\F_{p}(\zeta_{\ell})$.

Setting $\alpha_\ell=(1\otimes\eta\tilde\lambda)\alpha'_\ell\in A_\ell$ for $\eta\in\F_{p}(\zeta_\ell)$,
we get $\a_{\ell}=\eta^\ell(\embedcyc{l}{p^a-1})^{-1}(\stda_{p^a-1})=\eta^\ell\stda_\ell$, and
thus the standard $\alpha_\ell\in A_\ell$ are the $(1\otimes\zeta_\ell^u\tilde\lambda)\alpha'_\ell$, for $0\leq u<\ell$.
\end{proof}

\begin{definition}
\label{sstandard}
A generating element $s\in\F_{p^\ell}$ is called \emph{standard}
if it is of the form $s=\first{\alpha_{\ell}}{\zeta_\ell}$ for $\alpha_{\ell}\in A_\ell$
a standard solution of~\eqref{h90}.

The \emph{standard defining polynomial} $P_\ell$ for $\F_{p^\ell}$ is then the minimal polynomial over $\F_p$
of such a standard $s$.
\end{definition}

By Proposition~\ref{depend}, we note that $P_\ell$ is entirely determined by $\stda_{\ell}$,
and thus, by the given cyclotomic lattice $\system$, possibly up to order $p^{\nu(\ell)}-1$.
As an example, we give in Table~\ref{tab:std-polys} the first ten
standard polynomials induced by the system of Conway polynomials
for $p=2$ (thus, in this example, $P_\ell$ only depends on the Conway polynomial of degree $\nu(\ell)$).

We remark that it is easy to extend the definitions of decorated
algebras and standard elements to any extension degree, similarly to
the way this is done for the basic Lenstra--Allombert algorithm. %
Use any (standard) construction for Artin-Schreier towers over finite
fields (e.g.,~\cite{df+schost12}), define decorated algebras by
tensoring together Kummer algebras and Artin-Schreier extensions of
$\F_p$, and define standard elements as, e.g., the product of a
solution of multiplicative H90 and one of additive H90. %
While this solution is simple and effective, it is rather orthogonal
to our work, hence we omit the details here.

\begin{table}[b]
  \centering
  \small
  \begin{tabularx}{\columnwidth}{>{\raggedleft\arraybackslash}X | >{\raggedleft\arraybackslash}X}
    $x+1$ & $x^9+x^7+x^4+x^2+1$\\
    $x^3+x+1$ & $x^{11}+x^8+x^7+x^6+x^2+x+1$\\
    $x^5+x^3+1$ & $x^{13}+x^{10}+x^5+x^3+1$\\
    $x^7+x+1$ & $x^{15}+x+1$ \\[0.3em]
    \multicolumn{2}{r}{$x^{17}+x^{11}+x^{10}+x^8+ x^7+x^6+x^4+x^3+x^2+x+1$} \\
    \multicolumn{2}{r}{$x^{19}+x^{17}+x^{16}+x^{15}+x^{14}+x^{13}+x^{12}+x^8+x^7+x^6+x^5+x^3+1$}
  \end{tabularx}
  
  \caption{The first ten standard polynomials derived from Conway
    polynomials for $p=2$.}
  \label{tab:std-polys}
\end{table}

The decoration of an algebra $A_\ell$, and the associated standard
generating element and polynomial for $\F_{p^\ell}$, can be computed by the simple adaptation of
Allombert's algorithm presented below. %

\begin{algorithm}
  \caption{(Decoration -- Standardization)}
  \label{algo:decoration}
  \begin{algorithmic}[1]
    \REQUIRE {$\F_{p^\ell}$, for $\ell$ prime to $p$, and $\system$ a cyclotomic lattice.}
    \ENSURE {$(A_\ell,\alpha_\ell)$ decorated, $P_\ell$ standard irreducible polynomial of degree $\ell$,
      and $s\in\F_{p^\ell}$ standard generating element inducing $\F_{p^\ell}\simeq\F_{p}[T]/(P_\ell)$.}
  \STATE Prepare the Kummer algebra $A_\ell$.
  \STATE Prepare $\stda_\ell=(\embedcyc{l}{p^a-1})^{-1}((\zeta_{p^a-1})^a)\in\F_p(\zeta_\ell)$.
  \STATE Find $\alpha'_\ell\in A_\ell$ nonzero solution of~\eqref{h90} for $\zeta_\ell$.
  \STATE Compute its Kummer constant: $(\alpha'_\ell)^\ell=1\otimes\a'_\ell$.
  \STATE Compute $\kappa$ a $\ell$-th root of $\stda_\ell/\a'_\ell$.
  \STATE Set $\alpha_{\ell}=(1\otimes\kappa)\alpha'_\ell$.
  \STATE Compute $P_\ell$ the minimal polynomial of $\first{\alpha_{\ell}}{\zeta_\ell}$ over $\F_{p}$.
  \STATE Return $(A_\ell,\alpha_\ell)$, $P_\ell$, and $\first{\alpha_\ell}{\zeta_\ell}$.
  \end{algorithmic}
\end{algorithm}

Algorithm~\ref{algo:decoration} is correct, indeed Proposition~\ref{standardexiste} ensures that
a standard $\alpha_{\ell}$ exists, and thus $\alpha'_\ell=(1\otimes\kappa^{-1})\alpha_{\ell}$
for some $\kappa\in\F_{p}(\zeta_\ell)$, so $\stda_\ell/\a'_\ell=\kappa^\ell$ is a $\ell$-th power.

By design, decorated Kummer algebras of the same level admit standard Kummer embeddings,
under which the corresponding standard solutions of~\eqref{h90} are power-compatible:
\begin{proposition}
\label{embedincomplete}
Let $\ell\,|\,m$ be integers prime to $p$ and such that $\nu(\ell)=\nu(m)=a$.
Let $(A_\ell,\alpha_\ell)$, $(A_m,\alpha_m)$ be decorated Kummer algebras
of degree $\ell,m$ respectively (and of the same level $a$).
Then, there is a unique Kummer embedding
\[ \embedalg{\ell}{m}:A_\ell\hookrightarrow A_m \]
such that $\embedalg{\ell}{m}(\alpha_\ell)=(\alpha_m)^{\frac{m}{\ell}}$.
\end{proposition}
\begin{proof}
Proposition~\ref{Phialpha} with $\hat{\alpha}=(\alpha_m)^{\frac{m}{\ell}}$.
\end{proof}
Most often we will apply Proposition~\ref{embedincomplete} with $m=p^a-1$.

On the other hand, 
since power-compatibility implies $\a_m=\embedcyc{\ell}{m}(\a_\ell)$,
it cannot be satisfied for a Kummer embedding
between decorated Kummer algebras of different levels.
However, at least between \emph{complete} decorated algebras, we can request some \emph{norm}-compatibility
instead.

Let $A_m$ be a Kummer algebra of level $b=\nu(m)$, so
\[ A_m=\F_{p^m}\otimes\F_{p}(\zeta_m)\overset{\sim}{\longrightarrow}\F_{p^m}\otimes\F_{p}(\zeta_{p^b-1}), \]
where the isomorphism is given by $1\otimes\embedcyc{m}{p^b-1}$.
Then, for an integer $a\,|\,b$,
the subalgebra of $A_m$ invariant under $1\otimes\sigma^a$ is identified
by this isomorphism with
\[ (A_m)^{1\otimes\sigma^a}\simeq\F_{p^m}\otimes\F_{p}((\zeta_{p^b-1})^{\frac{p^b-1}{p^a-1}}), \]
where
$(\zeta_{p^b-1})^{\frac{p^b-1}{p^a-1}}=\operatorname{N}_{\F_{p^b}/\F_{p^a}}(\zeta_{p^b-1})=\embedcyc{p^a-1}{p^b-1}(\zeta_{p^a-1})$,
and with $\operatorname{N}_{\F_{p^b}/\F_{p^a}}$ the norm of the field extension
$\F_{p^b}/\F_{p^a}$.

\begin{definition}
\label{def:norm}
Given a Kummer algebra $A_n$, and some integers $a\,|\,b\,|\,\nu(n)$, we define
the \emph{scalar norm} operator
\[
\begin{array}{cccc}
  \norm_{b/a,A_n}: & (A_n)^{1\otimes\sigma^b} & \to & (A_n)^{1\otimes\sigma^a} \\
  & \gamma & \mapsto & \prod_{0\leq j<\frac{b}{a}} (1 \otimes \sigma^{ja})(\gamma).
\end{array}
\]
\end{definition}
This is well-defined, i.e., the image of $\norm_{b/a,A_n}$ is invariant under $1\otimes\sigma^a$ as specified.
Often the ambient algebra $A_n$ will be implicit,
and we will write $\norm_{b/a}$ instead of $\norm_{b/a,A_n}$.

By construction, $\norm_{b/a}$ acts on $1\otimes\F_{p^b}^\times$ as $1\otimes\operatorname{N}_{\F_{p^b}/\F_{p^a}}$.
Scalar norms are multiplicative:
\[
  \norm_{b/a}(\gamma\gamma')=\norm_{b/a}(\gamma)\norm_{b/a}(\gamma'),
\]
transitive: \[ \norm_{c/a} = \norm_{b/a}\circ\norm_{c/b},\] and they
commute with $\sigma\otimes1$.
\begin{proposition}
\label{embedcomplete}
Let $a\,|\,b$ be integers,
and let $(A_{p^a-1},\alpha_{p^a-1})$, $(A_{p^b-1},\alpha_{p^b-1})$ be decorated complete Kummer algebras
of level $a,b$ respectively.
Then there is a unique Kummer embedding
\[ \embedalg{p^a-1}{p^b-1}:A_{p^a-1}\hookrightarrow A_{p^b-1} \]
such that $\embedalg{p^a-1}{p^b-1}(\alpha_{p^a-1})=\norm_{b/a}(\alpha_{p^b-1})$.
\end{proposition}
\begin{proof}
By Lemma~\ref{FrobFrob} and the properties of the norm,
\begin{equation*}
\begin{split}
(\norm_{b/a}(\alpha_{p^b-1}))^{p^a}&=(\sigma^a\otimes\sigma^a)(\norm_{b/a}(\alpha_{p^b-1}))\\
&=(\sigma^a\otimes1)(\norm_{b/a}(\alpha_{p^b-1}))\\
&=\norm_{b/a}((\sigma^a\otimes1)(\alpha_{p^b-1}))\\
&=\norm_{b/a}((1\otimes(\zeta_{p^b-1})^a)\alpha_{p^b-1})\\
&=(1\otimes\embedcyc{p^a-1}{p^b-1}(\zeta_{p^a-1})^a)\norm_{b/a}(\alpha_{p^b-1}).
\end{split}
\end{equation*}
So $\hat\alpha=\norm_{b/a}(\alpha_{p^b-1})$
satisfies $(\hat{\alpha})^{p^a-1}=1\otimes\embedcyc{p^a-1}{p^b-1}(\a_{p^a-1})$
and we conclude with Proposition~\ref{Phialpha}.
\end{proof}

\section{Standard embeddings}
\label{sec:construction}

In Proposition~\ref{embedincomplete}, we saw how to construct a standard
power-compatible embedding of a decorated Kummer algebra into
its decorated complete algebra,
and in Proposition~\ref{embedcomplete}, a standard
norm-compatible embedding between decorated complete algebras
of dividing levels.
 
Now, consider general $\ell\,|\,m$ not divisible by $p$,
set $a=\nu(\ell)$, $b=\nu(m)$, and consider the diagram
\begin{equation*}
\label{3cotes}
\begin{CD}
(A_{p^a-1},\alpha_{p^a-1}) @>{\embedalg{p^a-1}{p^b-1}}>> (A_{p^b-1},\alpha_{p^b-1}) \\
@A{\embedalg{\ell}{p^a-1}}AA @AA{\embedalg{m}{p^b-1}}A\\
(A_\ell,\alpha_\ell) @. (A_m,\alpha_m)
\end{CD}
\end{equation*}
of standard embeddings of decorated algebras.
\begin{lemma}
In this setting, there exists a unique Kummer embedding
\[ \embedalg{\ell}{m}:A_\ell\hookrightarrow A_m \]
that makes the diagram commute.
\end{lemma}
\begin{proof}
Consider $\hat{\hat\alpha}=\embedalg{p^a-1}{p^b-1}(\embedalg{\ell}{p^a-1}(\alpha_\ell))\in A_{p^b-1}$.
Then $\hat{\hat\alpha}$ is invariant under $\sigma^\ell\otimes1$ and
under $1\otimes\sigma^a$ (because $\alpha_\ell$ is),
thus, \textit{a fortiori}, invariant under $\sigma^m\otimes1$ and
under $1\otimes\sigma^b$, which means it lies in the image
of $\embedalg{m}{p^b-1}$.
We can then set $\hat\alpha=(\embedalg{m}{p^b-1})^{-1}(\hat{\hat\alpha})$.
If $\embedalg{\ell}{m}$ exists, then it necessarily maps $\alpha_\ell$
to $\hat\alpha$.
However, chasing in the diagram, it is easily seen that $\hat\alpha$ is
a solution of~\eqref{h90} for $(\zeta_m)^{\frac{m}{\ell}}$ that
satisfies $(\hat{\alpha})^\ell=1\otimes\embedcyc{\ell}{m}(\a_\ell)$,
and we conclude with Proposition~\ref{Phialpha}.
\end{proof}
This existence result is ``constructive'', but impractical,
since it requires computations in the possibly very large algebra $A_{p^b-1}$.
However, as in Algorithm~\ref{algo:allombert}, one should be able
to write $\hat{\alpha}=(1\otimes\kappa)(\alpha_m)^{\frac{m}{\ell}}$
for some $\kappa\in\F_{p}(\zeta_m)$. Moreover $\hat{\alpha}$ is uniquely
determined by our data, thus, so should $\kappa$.
Now our aim is to give an explicit expression for this $\kappa=\kappa_{\ell,m}$.
We start with the case of complete algebras.
\begin{proposition}
\label{key}
In the complete algebra $A_{p^b-1}$ we have
\[ (\alpha_{p^b-1})^{\frac{p^b-1}{p^a-1}}=(1\otimes\zeta_{p^b-1})^{\frac{(b-a)p^{b+a}-bp^b+ap^a}{(p^a-1)^2}}\norm_{b/a}(\alpha_{p^b-1}).\]
\end{proposition}
\begin{proof}
Using first Lemma~\ref{FrobFrob}, and then~\eqref{h90}, we get:
\begin{equation*}
\begin{split}
\frac{(\alpha_{p^b-1})^{\frac{p^b-1}{p^a-1}}}{\norm_{b/a}(\alpha_{p^b-1})}&=\prod_{0\leq j<\frac{b}{a}}\frac{(\sigma^{ja}\otimes\sigma^{ja})(\alpha_{p^b-1})}{(1\otimes\sigma^{ja})(\alpha_{p^b-1})}\\
&=\prod_{0\leq j<\frac{b}{a}}(1\otimes\sigma^{ja})\left(\frac{(\sigma^{ja}\otimes 1)(\alpha_{p^b-1})}{\alpha_{p^b-1}}\right)\\
&=\prod_{0\leq j<\frac{b}{a}}(1\otimes\sigma^{ja})(1\otimes\zeta_{p^b-1})^{ja}\\
&=(1\otimes\zeta_{p^b-1})^{\sum_{0\leq j<\frac{b}{a}}jap^{ja}}.
\end{split}
\end{equation*}
We conclude thanks to the identity
\[ \sum_{0\leq j<n}jT^j=T\frac{d}{dT}\!\left(\frac{T^n-1}{T-1}\right)=\frac{(n-1)T^{n+1}-nT^n+T}{(T-1)^2}. \]
\end{proof}
\begin{corollary}
\label{explicit_general_standard_embedding}
Let $(A_\ell,\alpha_\ell)$ and $(A_m,\alpha_m)$ be decorated Kummer algebras,
of respective degrees $\ell\,|\,m$ prime to $p$.
Then the standard Kummer embedding $\embedalg{\ell}{m}:A_\ell\hookrightarrow A_m$
is defined by the assignation
$\alpha_\ell\;\mapsto\;(1\otimes\kappa_{\ell,m})(\alpha_m)^{\frac{m}{\ell}}$, where
\[
  \kappa_{\ell,m}=(\embedcyc{m}{p^b-1})^{-1}((\zeta_{p^b-1})^{-\frac{(b-a)p^{b+a}-bp^b+ap^a}{(p^a-1)\ell}}).
\]
\end{corollary}
\begin{proof}
It suffices to check that $\embedalg{p^a-1}{p^b-1}(\embedalg{\ell}{p^a-1}(\alpha_\ell))$
and the image of the right-hand-side under $\embedalg{m}{p^b-1}$ coincide in $A_{p^b-1}$.
However, we have $\embedalg{p^a-1}{p^b-1}(\embedalg{\ell}{p^a-1}(\alpha_\ell))=\norm_{b/a}(\alpha_{p^b-1})^{\frac{p^a-1}{\ell}}$,
while $\embedalg{m}{p^b-1}((\alpha_m)^{\frac{m}{\ell}})=(\alpha_{p^b-1})^{\frac{p^b-1}{\ell}}$,
and we conclude with Proposition~\ref{key}
\end{proof}

\begin{proposition}
\label{standard_K_embeddings_compatibles}
Standard Kummer embeddings are compatible with composition:
if $(A_\ell,\alpha_\ell)$, $(A_m,\alpha_m)$, and $(A_n,\alpha_n)$ are decorated Kummer algebras
with $\ell\,|\,m\,|\,n$, the corresponding standard embeddings
satisfy $\embedalg{\ell}{n}=\embedalg{m}{n}\circ\embedalg{\ell}{m}$.
\end{proposition}
\begin{proof}
We have to show that $\embedalg{\ell}{n}(\alpha_\ell)=\embedalg{m}{n}(\embedalg{\ell}{m}(\alpha_\ell))$;
nothing but a pleasant computation with the explicit constants given 
by Corollary~\ref{explicit_general_standard_embedding}.

Alternatively, set $a=\nu(\ell),b=\nu(m),c=\nu(n)$, and decorate $A_{p^a-1},A_{p^b-1},A_{p^c-1}$.
It suffices to show that the elements $\embedalg{\ell}{n}(\alpha_\ell)$
and $\embedalg{m}{n}(\embedalg{\ell}{m}(\alpha_\ell))$ have the same image under $\embedalg{n}{p^c-1}$
in $A_{p^c-1}$.
Chasing in the diagram
\begin{equation*}
\begin{CD}
A_{p^a-1} @>>> A_{p^b-1} @>>> A_{p^c-1} \\
@AAA @AAA @AAA\\
A_\ell @>>> A_m @>>> A_n
\end{CD}
\end{equation*}
we see that this common image
is $\norm_{c/a}(\alpha_{p^c-1})^{\frac{p^a-1}{\ell}}$.
\end{proof}

By a \emph{decorated finite field} (of degree $\ell$, an integer prime to $p$,
and relative to a given cyclotomic lattice $\system$),
we mean a pair $(\F_{p^\ell},s_\ell)$,
where $\F_{p^\ell}$ is a finite field,
and $s_\ell\in\F_{p^\ell}$ a standard generating element
in the sense of Definition~\ref{sstandard}.

We can finally state:
\begin{algorithm}
  \caption{(Standard compatible embeddings)}
  \label{algo:std_embed}
  \begin{algorithmic}[1]
    \REQUIRE {$\system$ a cyclotomic lattice, and $(\F_{p^\ell},s_\ell)$, $(\F_{p^m},s_m)$, decorated finite fields, for $\ell\,|\,m$ integers prime to $p$.}
    \ENSURE {$t\in\F_{p^m}$, such that the assignation $s_\ell\mapsto t$
      defines a standard embedding $\embed{\ell}{m}:\F_{p^\ell}\hookrightarrow\F_{p^m}$,
      compatible with composition.}
  \STATE Prepare the Kummer algebras $A_\ell$ and $A_m$.
  \STATE Recover $\alpha_\ell$ from $s_\ell$ and $\alpha_m$ from $s_m$ using equations~\eqref{recoveralpha}.
  \STATE Compute $\kappa_{\ell,m}=(\embedcyc{m}{p^b-1})^{-1}((\zeta_{p^b-1})^{-\frac{(b-a)p^{b+a}-bp^b+ap^a}{(p^a-1)\ell}})$ where $a=\nu(\ell)$, $b=\nu(m)$.
  \STATE Return $\first{(1\otimes\kappa)(\alpha_m)^{\frac{m}{\ell}}}{(\zeta_m)^{\frac{m}{\ell}}}$.
  \end{algorithmic}
\end{algorithm}
\begin{proposition}
\label{standard_ff_embeddings_compatibles}
Standard finite field embeddings are compatible with composition:
if $(\F_{p^\ell},\alpha_\ell)$, $(\F_{p^m},\alpha_m)$, and $(\F_{p^n},\alpha_n)$ are decorated finite fields
with $\ell\,|\,m\,|\,n$, the corresponding standard embeddings
satisfy $\embed{\ell}{n}=\embed{m}{n}\circ\embed{\ell}{m}$.
\end{proposition}
\begin{proof}
Corollary~\ref{alphaphi} and Proposition~\ref{standard_K_embeddings_compatibles}.
\end{proof}

\section{Implementation}
\label{sec:implementation}

In the previous sections we kept the description of Kummer embeddings
abstract, leaving many computational details unspecified. %
There are various ways in which our algorithms can possibly be
implemented, depending on how one chooses to represent finite fields
and the cyclotomic lattice $\system$. %
A reasonable option is to use (pseudo)-Conway polynomials to represent
the fields $\F_p(\zeta_{p^a-1})$, and deduce from them the smallest
possible representation for any other field $\F_p(\zeta_\ell)$. %
Assuming this technique, we can prove a bound on the complexity of our
algorithms.

\begin{proposition}
  Given a collection of (pseudo)-Conway polynomials for $\F_p$, of
  degree up to $d$, standard solutions $\alpha_\ell$ of~\eqref{h90} can
  be computed for any $\ell\,|\,(p^i-1)$ for any $i\le d$ using
  $O\bigr(\M(\ell^2)\log(\ell) + \M(\ell)\log(\ell)\log(p)\bigl)$ %
  operations. %
  After that, Kummer embeddings $\F_{p^\ell}\subset\F_{p^m}$ can be
  computed using $O(\M(m^2)\log(m))$ operations.
\end{proposition}
\begin{proof}
  Let $a=\nu(\ell)$ be the level of $A_\ell$. We take the $a$-th
  polynomial from the collection of (pseudo)-Conway polynomials, and
  use it to define $\zeta_{p^a-1}$. %
  Because $a\in O(\ell)$, the cost of multiplications in
  $\F_p(\zeta_{p^a-1})$ will be bounded by $O(\M(\ell))$.

  From $\zeta_{p^a-1}$, we compute $\zeta_l$ using $O(\ell\M(\ell))$
  operations, and its minimal polynomial in
  $O(\ell^{(\omega+1)/2})$. %
  Then, the Kummer constant $\stda_\ell=(\zeta_{p^a-1})^a$ is computed
  in negligible time, and its expression in the power basis of
  $\zeta_\ell$ is computed in $O(\ell^{(\omega+1)/2})$ using the
  algorithms for evaluating embeddings mentioned in
  Section~\ref{sec:conway}.

  To construct the Kummer algebra
  $A_\ell=\F_{p^\ell}\otimes\F_p(\zeta_\ell)$ we need an irreducible
  polynomial of degree $\ell$, not necessarily related to the
  (pseudo)-Conway polynomials used to represent the fields of
  scalars. %
  Very efficient, quasi-optimal algorithms for finding such a
  polynomial are given
  in~\cite{BoFlSaSc06,couveignes+lercier11,DeDoSc13}, we can thus
  neglect this cost.
  
  The cost of computing a solution $\alpha'_\ell$ to~\eqref{h90} was
  extensively studied in~\cite{brieulle2018computing}, where it was
  found to be bounded by
  $O\bigr(\M(\ell^2)\log(\ell) + \M(\ell)\log(p)\bigl)$. %
  Then, the constant $\a'_l=(\alpha'_\ell)^\ell$ is computed using
  $O\left(\M(\ell^2)\log(\ell)\right)$ operations, and the $\ell$-th root
  $\kappa$ is computed in $O(\M(\ell)\log(\ell)\log(p))$ according
  to~\cite{brieulle2018computing}. %
  $\alpha_\ell=(1\otimes\kappa)\alpha'_\ell$ is then computed in a
  negligible number of operations.

  Finally, the projection $\first{\alpha_\ell}{\zeta_\ell}$ comes for
  free, and its minimal polynomial $P_\ell$ is computed again in
  $O(\ell^{(\omega+1)/2})$ operations.

  Now, in order to compute a Kummer embedding of $(\F_{p^\ell},\alpha_\ell)$
  into $(\F_{p^m},\alpha_m)$, we compute the scalar $\kappa_{\ell,m}$
  in $O(m\M(m))$ operations, and $(\alpha_m)^{\frac{m}{\ell}}$ in
  $O(\M(m^2)\log(m))$. %

  We then need to convert $(\alpha_m)^{\frac{m}{\ell}}$ in the power
  basis of $\zeta_\ell$. %
  Applying a generic change of basis algorithm as before would be too
  expensive: indeed we have to convert $m$ coefficients from the field
  of scalars $\F_p(\zeta_m)$ to $\F_p(\zeta_\ell)$, which would cost
  $O(m^{(\omega+3)/2})$. %
  Instead we notice that we are only interested in the value
  $\first{(1\otimes\kappa)(\alpha_m)^{\frac{m}{\ell}}}{\zeta_\ell}$,
  therefore we proceed as follows.

  Let $\Tr$ denote the trace map from $\F_p(\zeta_m)$ to
  $\F_p(\zeta_\ell)$, and let $\eta\in\F_p(\zeta_m)$ be such that
  $\Tr(\eta)=1$. %
  Then the map $x\mapsto \Tr(x\eta)$ sends
  $(\zeta_m)^{\frac{m}{\ell}}$ to $\zeta_\ell$, and is
  $\F_p(\zeta_\ell)$-linear, it thus agrees with the inverse map of
  $\zeta_\ell\mapsto(\zeta_m)^{\frac{m}{\ell}}$ on the image of
  $\F_p(\zeta_\ell)$. %
  
  We thus need to evaluate $x\mapsto\first{\Tr(x\eta)}{\zeta_\ell}$
  for many values in $\F_p(\zeta_m)$, but this is a $\F_p$-linear
  form, hence we can precompute its vector on the power basis of
  $\zeta_m$. %
  Let $h_m,h_\ell$ be the minimal polynomials of $\zeta_m,\zeta_\ell$,
  and let $b,a$ be their degrees. %
  Let $h_0$ be the constant coefficient of $h_\ell$, and let
  \[\tau = -\frac{h_0}{(\zeta_m)^{\frac{m}{\ell}}} \frac{h_m'(\zeta_m)}{h_\ell'((\zeta_m)^{\frac{m}{\ell}})}\in\F_p(\zeta_m),\]
  direct calculation shows that
  \begin{equation*}
    \sum_{i=0}^{b-1} \first{\Tr(\zeta_m^i)}{\zeta_\ell}Z^i = \frac{\tau(Z^{-1})}{Zh_m(Z^{-1})}  \mod Z^b,
  \end{equation*}
  where by $\tau(Z)$ we mean $\tau\in\F_p(\zeta_m)$ seen as a
  polynomial in $\zeta_m$. %
  Hence, we can compute the vector of the linear form
  $x\mapsto\first{\Tr(x)}{\zeta_l}$ using only basic polynomial
  arithmetic and modular composition, i.e., in $O(m^{(\omega+1)/2})$
  operations.

  Finally, we compute
  $(1\otimes\kappa\eta)(\alpha_m)^{\frac{m}{\ell}}$, we see it as a
  polynomial with coefficients in $\F_p(\zeta_m)$, and we apply the
  map $\first{\Tr(x)}{\zeta_l}$ to each coefficient. %
  This costs $O(m\M(m))$ operations.
\end{proof}

We remark that storing the decorated fields $(\F_{p^l},\alpha_\ell)$
requires $O(\ell^2)$ field elements, however, using the formulas
in~\cite{Allombert02,brieulle2018computing}, it is possible to only
store $\first{\alpha_\ell}{\zeta_\ell}$, and recover all other
coefficients of $\alpha_\ell$ in $O(\ell\M(\ell)\log(p))$ operations.

To demonstrate the feasibility of this approach, we implemented it in
the Julia-based CAS Nemo~\cite{Fieker:2017:NCA:3087604.3087611}, with
performance critical routines written in C/Flint~\cite{flint}. %
Our code is available as a Julia package at
\url{https://github.com/erou/LatticeGFH90.jl}.

We tested Algorithms~\ref{algo:decoration} and~\ref{algo:std_embed}
for various small primes, using precomputed Conway polynomials
available in Nemo. %
We do not see major differences between different primes. %
In Figure~\ref{fig:timings} we report timings obtained for the case
$p=3$, on an Intel Core i7-7500U CPU clocked at 2.70GHz, using Nemo
0.11.1 running on Julia 1.1.0, and Nemo's current version of Flint. %
The plot on the left shows timings for
Algorithm~\ref{algo:decoration}, for degrees $\ell$ growing from $1$
to $200$, for every $\ell$ coprime to $p$ and such that the
$\nu(\ell)$-th Conway polynomial is available in Nemo; the color scale
shows the level of the associated algebra $A_\ell$. %
The bottleneck of this algorithm appears to be the $\ell$-th root
extraction routine.

The plot on the right shows timings for
Algorithm~\ref{algo:std_embed}, measured by computing the standard
embedding of $\mathbb{F}_{p^2}$ in $\mathbb{F}_{p^\ell}$. %
As expected, computing the embeddings takes negligible time in
comparison to the decoration of the finite fields. %
We also tested embedding fields larger than $\F_{p^2}$, and noticed
that the running time mostly depends on the size of the larger field.

\begin{figure}[h]
  \centering
  \includegraphics[scale=0.84]{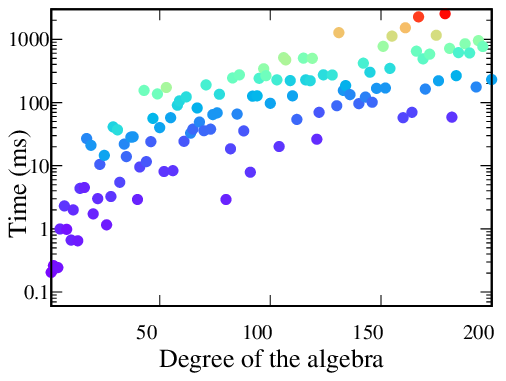}
  \includegraphics[scale=0.84]{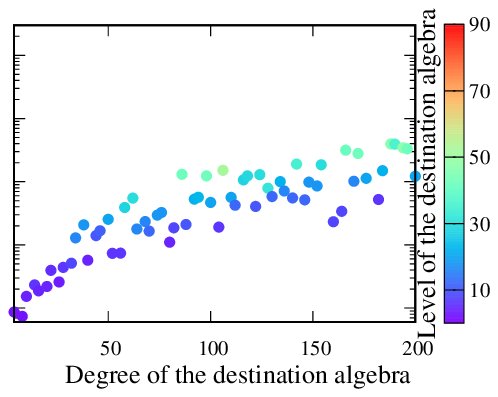}
  \caption{Timings for computing decorated fields
    $(\F_{p^\ell},\alpha_\ell)$ (left, log scale), and for computing the
    standard Kummer embedding from $\mathbb{F}_{p^2}$ to $\F_{p^\ell}$
    (right) for $p=3$.}
  \label{fig:timings}
\end{figure}

\section{Conclusion and future work}
\label{sec:conclusion}

We presented a new family of standard compatible polynomials for
defining finite fields. %
Its construction being dependent on the availability of Conway
polynomials, it has, at the present moment, very little practical
impact; its existence is nevertheless remarkable in itself.

It is even evident that computing our standard polynomials is
essentially equivalent to computing Conway polynomials; indeed from
$\alpha_\ell$ one can immediately deduce $(\zeta_{p^a-1})^a$, and by
taking an $a$-th root (doable in polynomial time in $\ell$), deduce
$\zeta_{p^a-1}$ and the associated Conway polynomial. %
Hence, an efficient algorithm for computing our polynomials (for
arbitrary degrees) would imply an efficient algorithm to compute
Conway polynomials, which would be unexpected.

However, our proposed implementation is not the only possible way to
exploit our definitions. %
It would be interesting, indeed, to find some middle ground between
the flexibility of the Bosma--Steel--Cannon framework and the rigidity
of Conway polynomials, for example by lazily enforcing the conditions
required to have a standard solution of~\eqref{h90}, while
incrementally constructing the lattice of roots of unity. %

Another line of work would be to give a complete implementation of a
lattice of finite fields, not limited to extensions of degree coprime
to $p$. %
We leave these questions for future work.

\begin{acks}
We would like to thank the anonymous reviewers for their useful
comments. %
We thank \'Eric Schost for fruitful discussions and for
helping bootstrap this work during a visit by two of the authors to
the University of Waterloo. %
We acknowledge financial support from the French ANR-15-CE39-0013
project \emph{Manta}, the \emph{OpenDreamKit} Horizon 2020 European
Research Infrastructures project (\#676541), and from the French
Domaine d'Int\'er\^et Majeur \emph{Math'Innov}.
\end{acks}

\bibliographystyle{ACM-Reference-Format}
\balance
\bibliography{gf-h90}

\end{document}